\documentclass[reqno]{amsart}
\usepackage{amsfonts}
\usepackage{amsmath,amsthm,mathrsfs}
\usepackage{amssymb}
\usepackage{enumitem}
\usepackage{tikz}
\usepackage{bm}

\usepackage[colorlinks=true,linkcolor=blue,citecolor=blue,urlcolor=blue,pdfborder={0 0 0}]{hyperref}

\usepackage{soul}

\usepackage{graphicx}
\usepackage{subcaption}

\usepackage{xcolor}

\theoremstyle{plain}
\newtheorem{thm}{Theorem}[section]

\newtheorem{lemma}[thm]{Lemma}
\newtheorem{cor}[thm]{Corollary}

\theoremstyle{definition}

\usepackage{fancyvrb}

\newcommand{\e}{\varepsilon}
\newcommand{\C}{\mathbb{C}}

\newcommand{\R}{\mathbb{R}}

\DeclareMathOperator{\area}{\mathrm{area}}
\DeclareMathOperator{\diam}{\mathrm{diam}}

\DeclareMathOperator{\Prob}{\mathbb{P}}

\newcommand{\mcl}{\mathcal}

\newcommand{\DD}{\mathbb{D}}
\newcommand{\rad}{{\rm rad}}
\newcommand{\vel}{{\rm VEL}}
\newcommand{\Z}{\mathbb{Z}}
\DeclareMathOperator{\corr}{\mathrm{corr}}

\title[A combinatorial criterion for macroscopic circles]{A combinatorial criterion for macroscopic circles in planar triangulations}
\author{Ori Gurel-Gurevich, Daniel C. Jerison and Asaf Nachmias}

\AtEndDocument{\bigskip{\footnotesize%
\par
\textsc{Ori Gurel-Gurevich} \par
\textsc{Hebrew University of Jerusalem} \par
\textit{Email:} \texttt{Ori.Gurel-Gurevich@mail.huji.ac.il} \par
\addvspace{\medskipamount}

\textsc{Daniel C. Jerison} \par
\textsc{Department of Mathematical Sciences, Tel Aviv University, Tel Aviv 69978, Israel} \par
  \textit{Email:} \texttt{jerison@mail.tau.ac.il}, \texttt{dcjerison@gmail.com} \par
  \addvspace{\medskipamount}

\textsc{Asaf Nachmias} \par
\textsc{Department of Mathematical Sciences, Tel Aviv University, Tel Aviv 69978, Israel} \par
  \textit{Email:} \texttt{asafnach@tauex.tau.ac.il} \par
  \addvspace{\medskipamount}

}}

\begin{document}

\begin{abstract}
Given a finite simple triangulation, we estimate the sizes of circles in its circle packing in terms of Cannon's \cite{Cannon} vertex extremal length. Our estimates provide control over the size of the largest circle in the packing. We use them, combined with results from \cite{GMS17b}, to prove that in a proper circle packing of the discrete \emph{mating-of-trees} random map model of Duplantier, Gwynne, Miller and Sheffield \cite{DMS14,GMS17}, the size of the largest circle goes to zero with high probability.
\end{abstract}

\maketitle

\vspace{.2cm}

\section{Introduction}

Koebe's circle packing theorem \cite{K36} (see also \cite{NStFlour18,St05}) is a canonical and widely used method of drawing planar maps. Various geometric properties of the circle packing encode important probabilistic information of the map. For example, a landmark result of He and Schramm \cite{HeSc} states that a bounded degree one-ended triangulation is recurrent if and only if its circle packing has no accumulation points.

In this paper, we estimate the sizes of circles in the circle packing of a planar triangulation in terms of Cannon's \cite{Cannon} \emph{vertex extremal length}. In particular, we provide an if-and-only-if criterion for the property that all the circles in the packing are small. This property is fundamentally important and is believed to hold in all natural models of random planar maps. Proving it for a random simple triangulation on $n$ vertices is an important open problem (see \cite[Section 6]{LeGICM}).

We use our criterion together with estimates of \cite{GMS17b} to prove that the size of the largest circle in the circle packing of the discrete \emph{mating-of-trees} model of Duplantier, Gwynne, Miller and Sheffield \cite{DMS14,GMS17} goes to zero with high probability. When combined with the main theorem of \cite{GGJN19}, this shows that discrete analytic functions on the circle packing embedding of the mating-of-trees map approximate classical analytic functions on the domain of the circle packing.
See the discussion in Section \ref{sec:motivation}.

\subsection{Circle packing and vertex extremal length} \label{sec:CP-VEL}

A {\bf circle packing} of a simple connected planar map $G$ with vertex set $V$ is a collection $\mcl P = \{C_v\}_{v \in V}$ of circles in the plane with disjoint interiors such that $C_v$ is tangent to $C_u$ if and only if $u$ and $v$ form an edge. We further require that for each vertex $v$, the cyclic order of the circles tangent to $C_v$ agrees with the cyclic order of the neighbors of $v$ in $G$. Koebe's \emph{circle packing theorem} mentioned above asserts that any simple connected planar map has a circle packing. Furthermore, when the map is a triangulation, the packing is unique up to M\"obius transformations. See \cite{NStFlour18} for details.

A {\bf triangulation with boundary} is a finite simple connected planar map in which all faces are triangles except for the outer face whose boundary is a simple cycle. Every triangulation with boundary has a circle packing in which the circles corresponding to vertices of the outer face are internally tangent to the unit circle $\partial \DD = \{z : |z|=1\}$ and all other circles are contained in the unit disk $\DD = \{z : |z| < 1\}$ \cite[Claim~4.9]{NStFlour18}. We call this a ``circle packing in $\DD$.'' This packing is unique up to M\"obius transformations from $\DD$ onto itself. A {\bf rooted triangulation with boundary} is a pair $(G, \rho)$ where $G$ is a triangulation with boundary and $\rho$ is a vertex in $V = V(G)$ that does not belong to the outer face. For any such $(G, \rho)$, denote by $\mcl P_\rho = \{C_v\}_{v \in V}$ a circle packing of $G$ in $\DD$ so that the circle $C_\rho$ is centered at the origin. By the aforementioned uniqueness, this circle packing is unique up to rotations. Therefore, the radius of each $C_v$, which we denote by rad$_\rho(v)$, is well-defined.

Our bounds use the notion of \emph{vertex extremal length}, introduced by Cannon \cite{Cannon}. Let $G=(V,E)$ be a graph. Given a function $m:V \to [0,\infty)$ and a finite path $\gamma$ in $G$, we define the {\bf length} of $\gamma$ according to $m$ as
$$ \textrm{len}_m(\gamma) = \sum_{v \in \gamma} m(v) \, .$$
Given a set of finite paths $\Gamma$ in $G$ we define
$$ \textrm{len}_m(\Gamma) = \inf _{\gamma \in \Gamma} \textrm{len}_m(\gamma) \, .$$
The {\bf vertex extremal length} of a set of finite paths $\Gamma$ in $G$ is now defined to be
$$ \textrm{VEL}(\Gamma) = \sup _{m:V\to [0,\infty)} { \textrm{len}_m(\Gamma)^2 \over \area(m) } \, ,$$
where $\area(m)= \sum_{v\in V} m(v)^2$.

Suppose that $(G,\rho)$ is a rooted triangulation with boundary. Given a vertex $v \neq \rho$, we denote by $\Gamma_{\rho,v}$ the set of paths in $G$ starting and ending at $v$ that have winding number $1$ around $\rho$. Also, for any vertex $v$ we denote by $\Gamma_{v}^\partial$ the set of paths in $G$ starting at $v$ and ending at a vertex of the outer face. Our first main result is the following.

\begin{thm} \label{thm:criterion} Let $(G,\rho)$ be a rooted triangulation with boundary and set $V = V(G)$. For any $v \in V$ distinct from $\rho$,
\begin{equation}  {\rm VEL}(\Gamma_{\rho,v}) \leq 4 \,{\rm rad}_\rho(v)^{-1}  \, .  \label{eq:winding}
\end{equation}
Furthermore, for any $v \in V$,
\begin{eqnarray}
 {\rm VEL}(\Gamma_{v}^\partial) \leq 1 +  \frac{1}{2} \log\left( \frac{1}{\rad_\rho(v)} \right)\, .
 \label{eq:boundary}
\end{eqnarray}
\end{thm}

The next theorem provides a bound in the other direction. For convenience we impose the convention that $\mathrm{VEL}(\Gamma_{\rho,\rho}) = 0$.

\begin{thm}
\label{thm:converse}
Let $(G,\rho)$ be a rooted triangulation with boundary and set $V = V(G)$. Let $\mcl P_\rho = \{C_v\}_{v \in V}$ be a circle packing of $G$ in $\DD$ such that $C_\rho$ is centered at the origin. If all the circles in $\mcl P_\rho$ have radius at most $\e$, then
\begin{equation} \label{eq:converse}
\min_{v \in V} \Big ( {\rm VEL}(\Gamma_{\rho,v}) \vee {\rm VEL}(\Gamma_{v}^\partial) \Big ) \geq c \log(1/\e)
\end{equation}
for some universal constant $c > 0$.
\end{thm}

We say that a sequence $(G_n,\rho_n)$ of rooted triangulations with boundary has {\bf no macroscopic circles} if $\max _{v \in V(G_n)} \textrm{rad}_{\rho_n}(v) \to 0$ as $n\to\infty$.
Combining the two theorems above, we obtain an if-and-only-if criterion for the existence of macroscopic circles.

\begin{cor} Let $(G_n,\rho_n)$ be a sequence of rooted triangulations with boundary. Then  $(G_n,\rho_n)$ has no macroscopic circles if and only if
$$ \min_{v \in V(G_n)} \Big ( {\rm VEL}(\Gamma_{\rho_n,v}) \vee {\rm VEL}(\Gamma_{v}^\partial) \Big ) \underset{n\to\infty}{\longrightarrow} \infty \, .$$
\end{cor}

The vertex extremal length is an ``embedding-invariant'' quantity that encodes various geometric properties of embeddings in which each vertex corresponds to a cell whose squared diameter is comparable to its area. The relationship between squared diameter and area might be precise (as in the case of circle packing) or might hold only in a rough averaged sense. Either way, we can use vertex extremal length to infer that if one such embedding of a map does not possess macroscopic cells, then neither does the other. We carry out this procedure in Sections \ref{sec:intromating} and \ref{sec:matingproof} for the ``mating-of-trees'' random map model, which we now describe.

\subsection{Mating of trees}\label{sec:intromating}
The discrete ``mating-of-trees'' is a random map model constructed and studied by Duplantier, Gwynne, Miller and Sheffield \cite{DMS14,GMS17,GMS17b}. The model is parametrized by a real number $\gamma\in(0,2)$ and is constructed to be in the universality class of the $\gamma$-LQG surface. Given $\gamma \in (0,2)$ fixed, for each $\e > 0$ the model defines an infinite random planar triangulation $\mcl G^\e$ on the vertex set $\e \Z$. In this paper we will not use the definition of $\mcl G^\e$ directly, instead using properties of $\mcl G^\e$ that have been proved in \cite{DMS14,GMS17,GMS17b}. Nevertheless, for the sake of completeness we now provide the definition of $\mcl G^\e$ as it appears in \cite{GMS17b}.

Start with a standard two-sided planar Brownian motion which is at the origin at time $0$. Apply an appropriate linear transformation so that both coordinates of the resulting process $(L,R): \R \to \R^2$ are standard two-sided linear Brownian motions with correlation $\corr(L_t, R_t) = -\cos(\pi \gamma^2 / 4)$ for all $t \in \R$ except $t = 0$ (where $L_0 = R_0 = 0$). Given $\e > 0$, draw the planar map $\mcl G^\e$ on the vertex set $\e \Z$ by first drawing an edge between each pair of consecutive vertices $x, x+\e$ so that the union of all these edges is a horizontal line. Next, for each pair of vertices $x_1,x_2 \in \e \Z$ such that $x_2 > x_1 + \e$ and
\begin{equation} \label{eq:GMS-Brownian}
\left( \inf_{t \in [x_1 - \e, x_1]} L_t \right) \vee \left( \inf_{t \in [x_2 - \e, x_2]} L_t \right) \leq \inf_{t \in [x_1, x_2 - \e]} L_t \, ,
\end{equation}
draw an edge between $x_1$ and $x_2$ in the space below the horizontal line. Finally, for each pair of vertices $x_1,x_2 \in \e \Z$ such that $x_2 > x_1 + \e$ and \eqref{eq:GMS-Brownian} holds with $R_t$ in place of $L_t$, draw an edge between $x_1$ and $x_2$ in the space above the horizontal line. Figure \ref{fig:mated-crt-map} shows a geometric interpretation of this process. The map $\mcl G^\e$ is almost surely a triangulation, with each pair of vertices connected by at most two edges.

\begin{figure}
\centering
\includegraphics[width=0.8\textwidth]{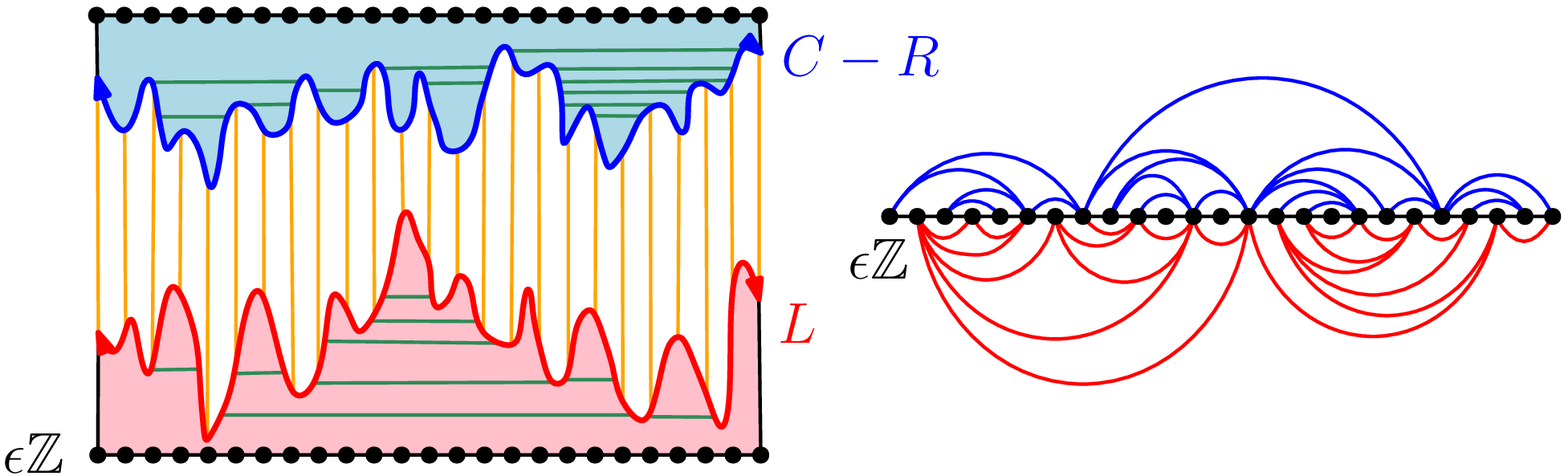}
\caption{To draw the map $\mcl G^\e$, choose a constant $C$ large enough that the graphs of the Brownian motions $L_t$ and $C-R_t$ do not intersect for $t$ in a given interval. The left image shows the graphs of $L$ and $C - R$ on the interval $[0, 12\e]$. Each $x \in \e \Z$ corresponds to the vertical strip $\{(t,y) : t \in [x-\e, x],\, y \in [L_t, C - R_t] \}$, which is bounded by orange vertical lines. Given $x_1,x_2 \in \e \Z$ with $x_2 > x_1 + \e$, equation \eqref{eq:GMS-Brownian} or its analogue for $R_t$ is satisfied if and only if one can draw a horizontal line between the strips corresponding to $x_1$ and $x_2$ that stays below $L$ or above $C-R$, respectively. Such horizontal lines are drawn in green (including some whose endpoints are outside the pictured interval). The right image is a drawing of $\mcl G^\e$ where the edges coming from green horizontal lines below $L$ are drawn in red and the edges coming from green horizontal lines above $C-R$ are drawn in blue. Images are from \cite{GMS17b} and are reproduced with permission.}
\label{fig:mated-crt-map}
\end{figure}

Many natural models of random planar maps can be encoded by a discrete version of the mating-of-trees map in which the Brownian motions $L,R$ are replaced by discrete time random walks. See \cite{GMS17b} for references. Thus the mating-of-trees map can be understood as a coarse-grained approximation of these models and can be used to derive new results about them. For instance, Gwynne and Miller \cite{GM18} have recently shown using this approach that the spectral dimension of the \emph{uniform infinite planar triangulation} (UIPT) is almost surely $2$.

A remarkable feature of the map $\mcl G^\e$ is that it comes with an \emph{a priori embedding} in $\C$. See \cite[Theorem 1.9]{DMS14} and \cite[Proposition 2.2]{GMS17}. This embedding plays a central role in the precise formulation of our small-circles result, Theorem \ref{thm:nomacroscopic}, and in its proof. It has the following properties. To each vertex $v$ of $\mcl G^\e$ there is an associated \emph{cell} $H_v^\e$, which is a compact connected subset of $\C$ such that the interiors of all cells are pairwise disjoint and vertices are adjacent in $\mcl G^\e$ if and only if their corresponding cells share a non-trivial connected boundary arc. The union of all cells is the whole plane. We remark that $\mcl G^\e$ has the same distribution for all $\e>0$ (by the scale-invariance of Brownian motion) but the a priori embedding is different. For a set $D \subset \C$ we write
$$ \mcl V \mcl G ^\e (D) = \{ v \in \mcl V \mcl G^\e : H_v^\e \cap D \neq \varnothing \} \, ,$$
where $\mcl V \mcl G^\e$ is the vertex set of $\mcl G^\e$. The set $\mcl V \mcl G ^\e (D)$ is finite if $D$ is bounded.

Let $B(0,\rho)$ denote the disk $\{z: |z| < \rho\}$. In \cite{GMS17b} various estimates are proven concerning the sets $\mcl V \mcl G^\e(B(0,\rho))$ for $\rho < 1$. For convenience, we scale the a priori embedding by a factor of $2$ so that the set $\mcl V \mcl G^\e(\DD)$ in our normalization is the same as the set $\mcl V \mcl G^\e(B(0,1/2))$ in the language of \cite{GMS17b}.

Under this scaling, let $\mcl G^\e(\DD)$ be the submap of $\mcl G^\e$ whose vertex set is $\mcl V \mcl G ^\e (\DD)$ and whose edges are the edges of $\mcl G^\e$ spanned by $\mcl V \mcl G ^\e (\DD)$. When $\e$ is small, with high probability the maximal cell diameter in $\mcl G^\e(\DD)$ is at most $\e^q$ for some constant $q > 0$ \cite[Lemma 2.7]{GMS17b}. In the construction that follows, we assume that this event occurs.

We make minor modifications to turn $\mcl G^\e(\DD)$ into a (simple) triangulation with boundary. These modifications are illustrated in Figure \ref{GMS-mods}. First, if $\mcl G^\e(\DD)$ is not already $2$-connected, we observe that all the vertices $v$ for which $H_v^\e$ intersects the slightly smaller disk $\{z : |z| < 1 - \e^q \}$ are contained in the same block (maximal $2$-connected component) of $\mcl G^\e(\DD)$. We restrict to this block, which we call $\mcl G_0^\e$. This cuts off the ``dangling ends'' associated with repeated vertices in the boundary of the outer face of $\mcl G^\e(\DD)$; the boundary of the new outer face is a simple cycle \cite[Proposition 4.2.5]{D17} whose vertices and edges are all part of the old boundary.

\begin{figure}
\centering
\includegraphics[width=0.7\textwidth]{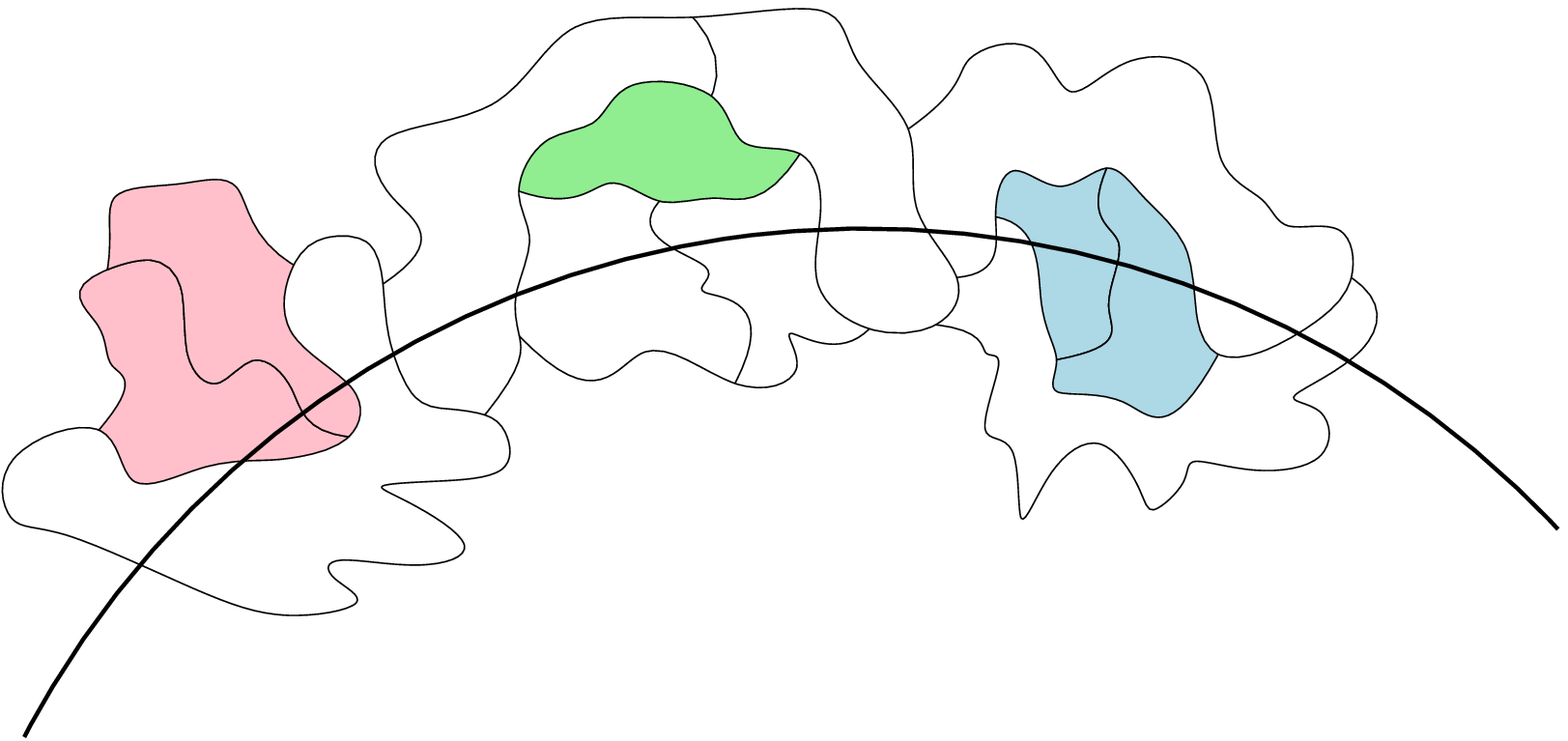}
\caption{Modification of $\mcl G^\e(\DD)$ into $\mcl G_2^\e$. Several cells $H_v^\e$ are shown along with a portion of the unit circle (thick black curve). The transition from $\mcl G^\e(\DD)$ to $\mcl G_0^\e$ deletes the vertices whose cells are colored in red. The transition from $\mcl G_0^\e$ to $\mcl G_1^\e$ adds the vertex whose cell is colored in green. Finally, the transition from $\mcl G_1^\e$ to $\mcl G_2^\e$ deletes the vertices whose cells are colored in blue and absorbs the blue area either into the cell above or into the cell below.}
\label{GMS-mods}	
\end{figure}

The map $\mcl G_0^\e$ might have inner faces which are not triangles because it is possible for a cycle in $\mcl G_0^\e$ to enclose vertices in $\mcl V \mcl G^\e \setminus \mcl V \mcl G_0^\e$. In this case we simply add to the map all of these enclosed vertices and their incident edges and denote the resulting map by $\mcl G_1^\e$. The boundary of the outer face is unchanged by this procedure.

Now the inner faces of $\mcl G_1^\e$ are all triangles, but the map is not necessarily simple. As described above, between any given pair of vertices there might be two edges (but no more than two). At each occurrence of such parallel edges we collapse them to a single edge and erase all the vertices and edges between them. Furthermore, in the embedding we assign the space of the erased cells arbitrarily to one of the two surrounding cells corresponding to the vertices of the double edge. This guarantees simplicity. We denote this map by $\mcl G_2^\e$ and observe that it is a triangulation with boundary according to the definition in Section \ref{sec:CP-VEL}. Indeed, $\mcl G_2^\e$ is a simple submap of $\mcl G_1^\e$ with the same outer face. Lastly, the vertex whose corresponding cell contains the origin (or any one of the vertices with this property, in case the origin is on a cell boundary) is declared to be the root and denoted by $\rho$. Since $\rho$ is not a vertex of the outer face, we have that $(\mcl G_2^\e, \rho)$ is a rooted triangulation with boundary.

As in Section \ref{sec:CP-VEL}, let $\mcl P^\e_\rho = \{C_v\}_{v \in \mcl V \mcl G_2^\e}$ be a circle packing of $\mcl G_2^\e$ in $\DD$ such that $C_\rho$ is centered at the origin. This packing is unique up to rotations. Denote the radius of each $C_v$ by $\rad^\e_\rho(v)$. The following theorem bounds the maximum radius in $\mcl P_\rho^\e$ with high probability. It implies that for a sequence $\e_n$ tending to $0$ sufficiently fast, the sequence of random rooted triangulations with boundary $(\mcl G_2^{\e_n}, \rho_n)$ almost surely has no macroscopic circles.

\begin{thm} \label{thm:nomacroscopic} There exist constants $C,c>0$, depending only on the parameter $\gamma \in (0,2)$, such that with probability at least $1-\e^c$,
$$ \max_{v\in \mcl V \mcl G_2^\e}\,\, \rad_\rho^\e (v) \leq {C \over \log(1/\e)} \, .$$
\end{thm}

Note that the low probability $\e^c$ encompasses both the event that the modification of $\mcl G^\e(\DD)$ into $\mcl G_2^\e$ fails due to the presence of a very large cell, and the event that the modification succeeds but the circle packing has a very large circle.

\subsection{Discrete complex analysis on random planar maps}\label{sec:motivation}

Discrete complex analysis has been pivotal in the study of statistical physics on two-dimensional lattices \cite{SmirnovPerc,SmirnovIsing,ChelkakSmirnov2011,ChelkakSmirnovIsing,SmirnovICM}. Recently, statistical physics on random planar maps has drawn a great deal of attention due to the conjectured KPZ correspondence, see \cite{Garban13,DS11}. Random planar maps are combinatorial objects and do not come equipped with a canonical embedding in the plane. Thus, a major challenge is to find an embedding on which discrete complex analysis can be performed. The motivation of the current paper and its companion \cite{GGJN19} is to address this challenge. Indeed, the combination of Theorem \ref{thm:nomacroscopic} and \cite[Theorem 1.1]{GGJN19} enables one to perform discrete complex analysis on the circle packing embedding of the mating-of-trees random map model. We provide a brief explanation here and refer the reader to \cite{GGJN19} for further details.

A natural candidate for the embedding of a generic planar map is the \emph{orthodiagonal} representation. An \emph{orthodiagonal map} is a plane graph having quadrilateral faces with orthogonal diagonals. It turns out that any simple 3-connected finite planar map, in particular any simple triangulation, can be represented by an orthodiagonal map via the circle packing theorem, see \cite[Section 2]{GGJN19}. Furthermore, Duffin \cite{D68} showed that orthodiagonal maps admit a very natural form of discrete analyticity: a complex-valued function $f$ is said to be {\bf discrete analytic} if for every inner quadrilateral $[v_1,w_1,v_2,w_2]$ of the map,
 \[
\frac{f(v_2) - f(v_1)}{v_2 - v_1} = \frac{f(w_2) - f(w_1)}{w_2 - w_1} \, .
\]
With this definition it can be shown that discrete contour integrals of discrete analytic functions vanish, and that the real part of a discrete analytic function is discrete harmonic with respect to positive real edge weights induced by the map, mirroring classical complex analysis theory. It is a natural and highly applicable question to ask whether such functions are close to continuous holomorphic functions---the answer is positive.

Indeed, following the work of \cite{D99,S13,W15}, in \cite{GGJN19} we prove a general result regarding the convergence of discrete harmonic functions on orthodiagonal maps to their continuous counterparts.\footnote{We emphasize that the discrete harmonic functions in \cite{GGJN19}, just as in \cite{D99,S13,W15}, are with respect to natural edge weights determined by the orthodiagonal map and not with respect to unit weights. See equation (1) in \cite[Section 1.1]{GGJN19}.} Our main contribution is to drop several local and global regularity conditions, such as bounded vertex degrees, present in the previous results \cite{D99,S13,W15}. These conditions have prevented such convergence results from applying to random planar maps.
By contrast, our convergence statement \cite[Theorem 1.1]{GGJN19} holds for any random map model that has an orthodiagonal representation with maximal edge length going to $0$. Theorem \ref{thm:nomacroscopic} shows that the circle packing in the unit disk $\DD$ of the mating-of-trees random triangulation $\mcl G_2^\e$ induces such an orthodiagonal representation. For further details, see \cite[Corollary 2.2]{GGJN19}. It follows that discrete harmonic and analytic functions on the circle packing embedding of $\mcl G_2^\e$ in $\DD$ converge to their continuous counterparts as $\e$ tends to $0$.

\section{Proofs} \label{sec:proofs}

\subsection{Proof of Theorem \ref{thm:criterion}}

We begin by proving \eqref{eq:winding}. By applying a rotation we may assume that the center of the circle $C_v$ lies on the positive $x$-axis, and we denote by $t_2 > t_1 > 0$ the two intersection points of $C_v$ with the $x$-axis. We have that $t_2 - t_1 = 2{\rm rad}_\rho(v)$. For each $t \in [t_1,t_2]$ let $C_t$ be the circle of radius $t$ around the origin (this is \emph{not} an element of the circle packing $\mcl P_\rho$). Let $\gamma_t$ denote the finite simple path in $G$ obtained from $C_t$ by starting at the point $(t,0)$ and the vertex $v$, then traversing $C_t$ counterclockwise concatenating to $\gamma_t$ the vertex of $G$ corresponding to any new circle of $\mcl P_\rho$ that we encounter, until we visit $C_v$ in the last step. See Figure \ref{fig:circlepath}. Note that for almost every $t \in [t_1,t_2]$ the circle $C_t$ does not intersect any tangency point of $\mcl P_\rho$, so $\gamma_t$ is well defined for this set of $t$'s. Since we start and end at $v$ and wind around the origin a single time, we have that $\gamma_t \in \Gamma_{\rho,v}$ for almost every $t\in[t_1,t_2]$.

Now, given $m:V \to [0,\infty)$ we have that
$$ \int_{t_1}^{t_2} {\rm len}_m(\gamma_t) dt = \sum_{u\in V} m(u) \int _{t_1}^{t_2} {\bf 1}_{\{u \in \gamma_t\}}dt \, .$$
The vertices $u$ that contribute to the sum are those for which the interior of $C_u$ intersects the annulus $A = \{z : t_1 \leq |z| \leq t_2\}$. Given such a circle $C_u$, label its closest point to the origin by $r_1 e^{i \varphi}$ and its farthest point from the origin by $r_2 e^{i \varphi}$ (the angles are the same). Then
\[
\int_{t_1}^{t_2} {\bf 1}_{\{u \in \gamma_t\}}dt = (r_2 \wedge t_2) - (r_1 \vee t_1) = 2\rad'_\rho(u)
\]
where $\rad'_\rho(u)$ is the radius of the circle $C'_u$ for which the line segment between $(r_1 \vee t_1)e^{i \varphi}$ and $(r_2 \wedge t_2)e^{i \varphi}$ is a diameter. One such circle $C'_u$ is illustrated in Figure \ref{fig:circlepath}. The circles $C'_u$ are all inside the annulus $A$, and they are internally disjoint because each $C'_u$ is contained inside $C_u$. Hence
\[
\sum_u (\rad'_\rho(u))^2 \leq t_2^2 - t_1^2 \leq 2(t_2 - t_1)
\]
and, by Cauchy-Schwarz,
\[
\begin{split}
\int_{t_1}^{t_2} {\rm len}_m(\gamma_t) dt = 2 \sum_u m(u) \rad'_\rho(u) &\leq 2 \left( \area(m) \sum_u (\rad'_\rho(u))^2 \right)^{1/2} \\
&\leq 2\Big( 2 \area(m) (t_2 - t_1) \Big)^{1/2}\,.
\end{split}
\]

We deduce that there must exist $t\in(t_1,t_2)$ such that
$$ {\rm len}_m(\gamma_t) \leq { 2 \sqrt{2\area(m)}  \over \sqrt{t_2-t_1} } = 2 \sqrt{\area(m) \over {\rm rad}_\rho(v)}  \, .$$
Plugging this into the definition of VEL$(\Gamma_{\rho,v})$ concludes the proof of \eqref{eq:winding}.

\begin{figure}
\centering
\begin{subfigure}{0.45\textwidth}
   \centering
   \includegraphics[width=1\textwidth]{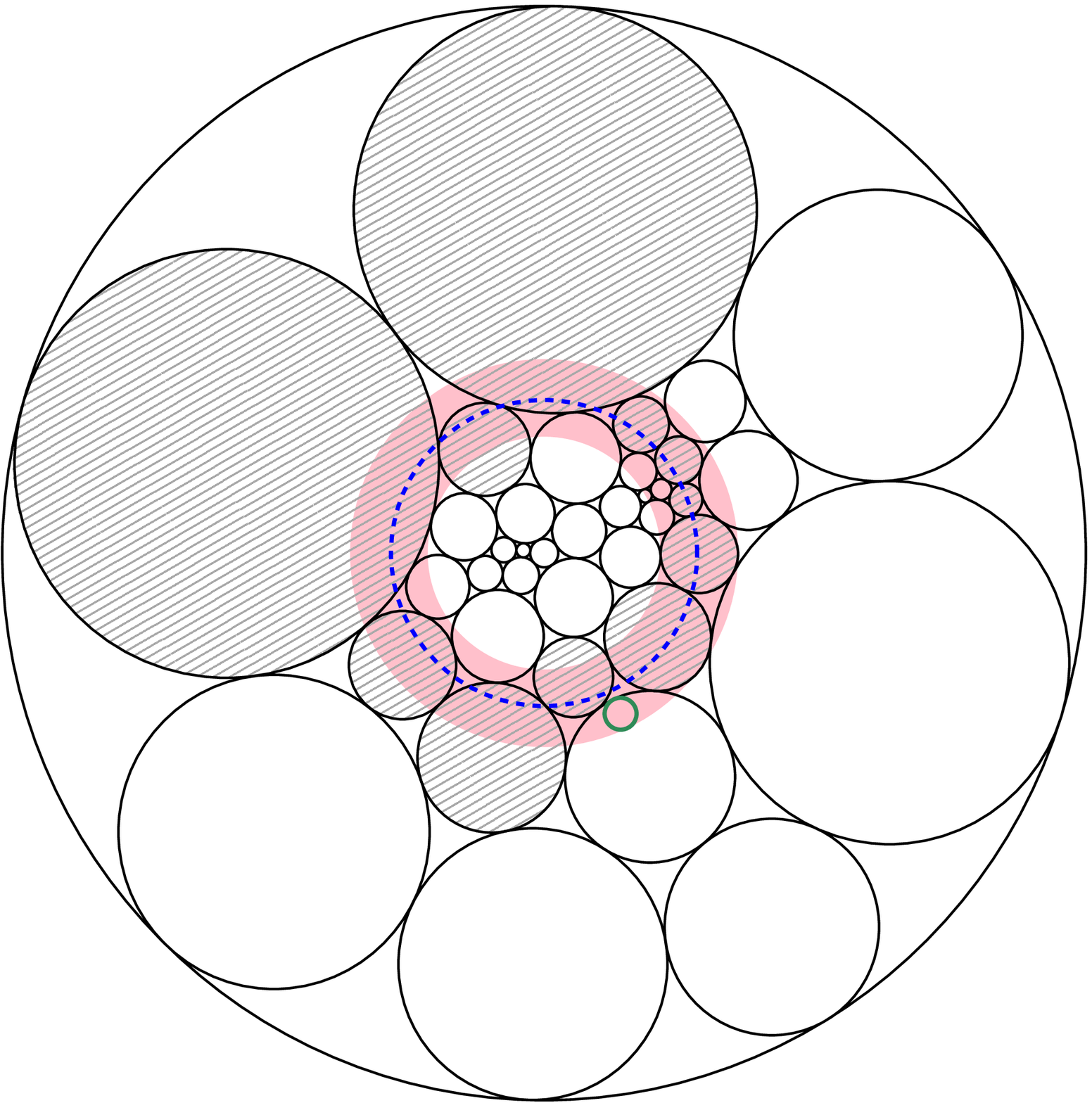}
   \caption{Construction of the path $\gamma_t$. The red annulus is $A = \{z : t_1 \leq |z| \leq t_2\}$ and the blue dashed circle is $C_t$. The circles corresponding to vertices in $\gamma_t$ are shaded in gray. Additionally, one of the circles $C'_u$ is drawn in green.}
   \label{fig:circlepath}
\end{subfigure}
\quad
\begin{subfigure}{0.45\textwidth}
   \centering
   \includegraphics[width=1\textwidth]{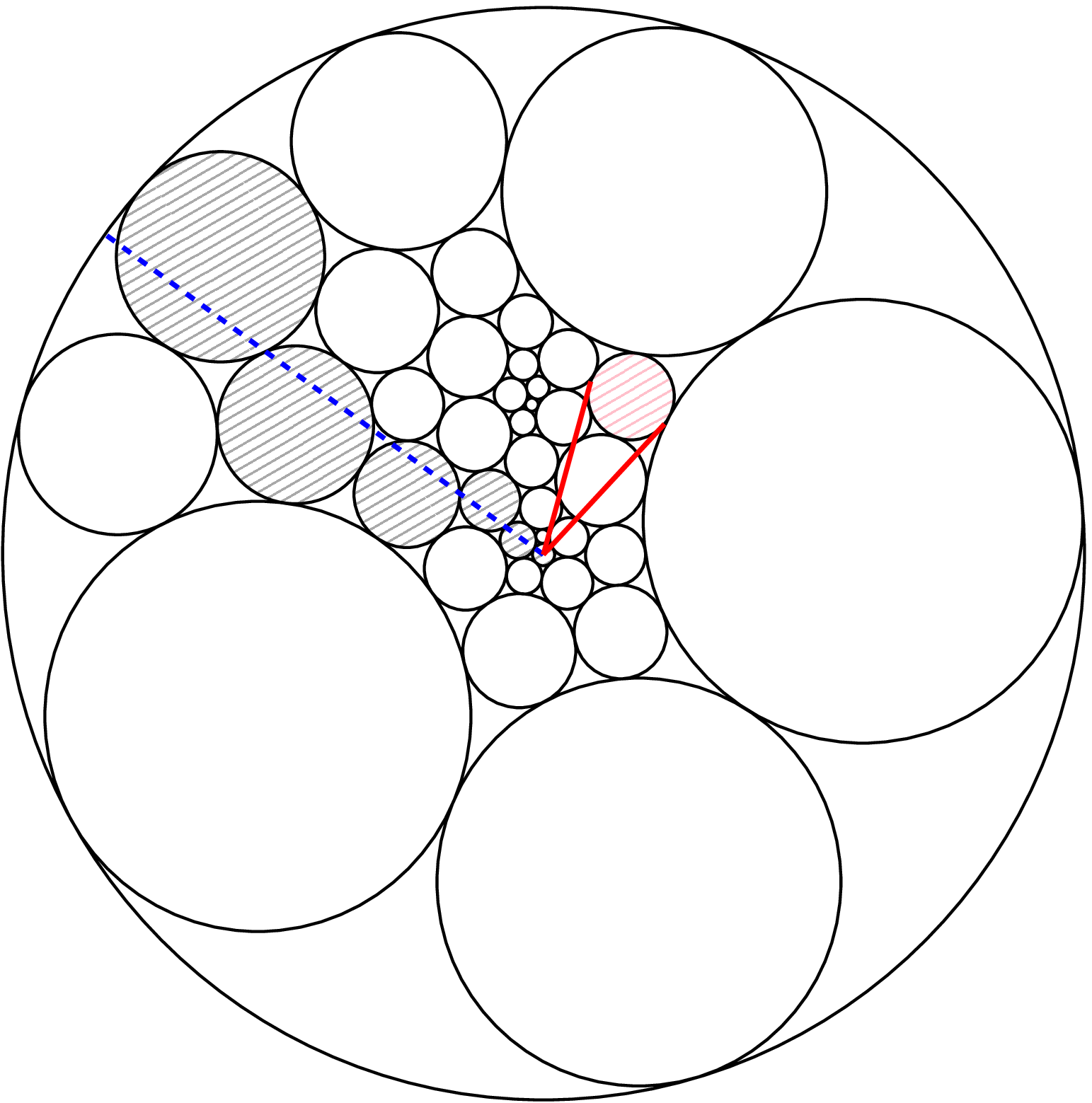}
   \caption{Construction of the path $\gamma_\theta$. The blue dashed line is $\ell_\theta$ and the circles corresponding to vertices in $\gamma_\theta$ are shaded in gray. The solid red segments are the tangent lines from the origin to a circle $C_u$ which is shaded in red.}
   \label{fig:diskpath}
\end{subfigure}
\caption{Illustrations of various parts of the proof of Theoeem \ref{thm:criterion}.}
\label{fig:VEL-paths}
\end{figure}

We now prove \eqref{eq:boundary}. If $v$ is a vertex of the outer face of $G$ then ${\rm VEL}(\Gamma_v^\partial) = 1$. In all other cases we will show that
\begin{equation} \label{eq:rad_vv}
{\rm VEL}(\Gamma_{v}^\partial) \leq 1 +  \frac{1}{2} \log\left( \frac{1}{\rad_v(v)} \right)\, .
\end{equation}
This implies \eqref{eq:boundary} because $\rad_\rho(v) < \rad_v(v)$ whenever $\rho \neq v$, as we now explain. To get from the packing $\mcl P_v$ to $\mcl P_\rho$, one applies a M\"{o}bius transformation of the form $z \mapsto e^{i\theta} (z + w)/(1 + \overline{w}z)$ for some $0 < |w| < 1$. This transforms a circle of radius $r < 1$ centered at the origin into a circle of radius $r(1 - |w|^2)/(1 - |w|^2 r^2) < r$.

We consider the packing $\mcl P_v = \{C_u\}_{u \in V}$, where $C_v$ is centered at the origin. For $\theta\in (0,2\pi)$ let $\ell_\theta$ be the straight line in the plane of angle $\theta$ starting from the origin and ending at $\partial \DD$. Let $\gamma_\theta$ be the path in $G$ obtained by taking the circles of $\mcl P_v$ that $\ell_\theta$ intersects according to the order in which they are intersected. See Figure \ref{fig:diskpath}. For almost every $\theta$ this is a finite simple path in $G$ which starts at $v$ and ends in one of the vertices of the outer face. As before, given any $m:V\to [0,\infty)$ we have that
$$ \int_{0}^{2\pi} {\rm len}_m(\gamma_\theta)d\theta = \sum_{u} m(u) \int_{0}^{2\pi} {\bf 1}_{\{u \in \gamma_\theta\}}d\theta \, .$$
For any $u\neq v$, the Lebesgue measure of $\theta$'s for which $u \in \gamma_\theta$ is precisely the angle between the two tangents to $C_u$ emanating from the origin, as shown in Figure \ref{fig:diskpath}. If $\alpha$ is half that angle, then $\sin(\alpha)=\rad_v(u)/d(u,v)$ where $d(u,v)$ is the Euclidean distance between the center of $C_u$ and the origin. It is clear that $\alpha \in (0,\pi/2)$ so that $\alpha \leq {\pi \over 2}\sin(\alpha)$, and we obtain
\[
\begin{split}
\int_{0}^{2\pi} {\rm len}_m(\gamma_\theta)d\theta &\leq 2\pi m(v) + \pi \sum_{u\neq v} \frac{m(u) \rad_v(u)}{d(u,v)} \\
&\leq 2\pi \sqrt{\area(m)} \left( 1 + \frac{1}{4}\sum_{u \neq v} \frac{(\rad_v(u))^2}{d(u,v)^2} \right)^{1/2}
\end{split}
\]
where the last inequality is Cauchy-Schwarz.

To bound the sum, let $D_u$ denote the interior of $C_u$. The function $f(x,y) = 1/(x^2 + y^2)$ on $\R^2 \setminus \{0\}$ is subharmonic (this can be checked by computing the Laplacian) and so its value at the center of $C_u$ is less than its average value on $D_u$. In other words,
\[
\frac{1}{d(u,v)^2} \leq \frac{1}{\pi \cdot (\rad_v(u))^2} \iint_{D_u} \frac{1}{x^2 + y^2} \, dx\,dy \,.
\]
The disks $D_u$ for $u \neq v$ are disjoint and all contained in the annulus $A' = \{z : \rad_v(v) \leq |z| \leq 1 \}$. Hence
\[
\begin{split}
\sum_{u \neq v} \frac{(\rad_v(u))^2}{d(u,v)^2} &\leq \sum_{u \neq v} \frac{1}{\pi} \iint_{D_u} \frac{1}{x^2 + y^2} \, dx\,dy \leq \frac{1}{\pi} \iint_{A'} \frac{1}{x^2 + y^2} \, dx\,dy \\
&= 2 \int_{\rad_v(v)}^1 \frac{1}{r} \, dr = 2 \log\left(\frac{1}{\rad_v(v)} \right) \, .
\end{split}
\]

We conclude that
\[
\frac{1}{2\pi} \int_{0}^{2\pi} {\rm len}_m(\gamma_\theta)d\theta \leq \sqrt{\area(m)} \sqrt{1 + \frac{1}{2} \log\left(\frac{1}{\rad_v(v)}\right) }
\]
so there is $\theta \in (0,2\pi)$ such that ${\rm len}_m(\gamma_\theta)$ has the same bound, proving \eqref{eq:rad_vv}. \qed

\subsection{Proof of Theorem \ref{thm:converse}} \label{sec:converse-proof}
Because all the paths in $\Gamma_{\rho,v}$ and $\Gamma_v^\partial$ pass through $v$, the left side of \eqref{eq:converse} is at least $1$ and we may assume that $\e$ is sufficiently small.

For $z\in \DD$ and $r_2>r_1>0$, we denote by $A_z(r_1,r_2)$ the annulus $\{w: r_1 < |z-w| < r_2\}$. Set $K = \lfloor \log_2(1/5\e) \rfloor - 2$. For $0 \leq k \leq K$, let $r_k = 2^k (5\e)$, so that $r_0 = 5\e$ and $r_K \in [1/8, 1/4]$. Given $v \in V$, let $c_v$ be the center of the circle $C_v \in \mcl P_\rho$ and let $B_k$ be the set of vertices $u$ for which $C_u \subset A_{c_v}(r_{k-1},r_{k})$. We define $m: V \to [0,\infty)$ by
\[
m(u) = \begin{cases} \rad_\rho(u)/r_k & \text{if } u\in B_k \text{ for some } 1 \leq k \leq K \\
0 & \text{otherwise}. \end{cases}
\]

For each $1\le k \le K$ we have
\[
\sum_{u\in B_k} m(u)^2 = \frac{1}{\pi r_k^2} \sum_{u\in B_k} \pi (\rad_\rho(u))^2 \le 1 \, ,
\]
hence $\area(m) \leq K$.

If $|c_v| \geq 1/2$ then every path $\gamma \in \Gamma_{\rho,v}$ must visit a vertex $v'$ with $|c_{v'} - c_v| \geq 1/4 \geq r_K$. If $|c_v| \leq 1/2$ then the same property holds for every path $\gamma \in \Gamma_v^\partial$.

Suppose $\gamma$ is a path in $G$ that starts at $v$ and visits a vertex $v'$ such that $|c_{v'} - c_v| \geq r_K$. The packing $\mcl P_\rho$ induces an embedding of $G$ in the plane where each vertex $u$ is drawn at $c_u$ and the edges $(u,u')$ are drawn as the straight lines $c_u c_{u'}$. Under this embedding, the path $\gamma$ becomes a piecewise linear curve in $\DD$ that crosses from the inside to the outside of each annulus $A_{c_v}(r_{k-1},r_k)$. We abuse notation by referring to this curve also as $\gamma$. If $\gamma$ goes through a vertex $u\in B_k$, then the Euclidean length of $\gamma \cap C_u$ is $2 \rad_\rho(u)$ while the contribution of $u$ to $\mathrm{len}_m(\gamma)$ is $\rad_\rho(u)/r_k$, thus the ratio between them is $1/2r_k$. As all the circles in $\mcl P_\rho$ have radius at most $\e$, any circle $C_u$ whose interior intersects $A_{c_v}(r_{k-1} + 2\e, r_k - 2\e)$ must be contained entirely in $A_{c_v}(r_{k-1},r_k)$. It follows that the contribution of $B_k$ to $\mathrm{len}_m(\gamma)$ is at least
\[
[(r_k - 2\e) - (r_{k-1} + 2\e)] / 2r_k \geq 1/20 \, .
\]
Therefore $\mathrm{len}_m(\gamma) \geq K/20$ and so $(\mathrm{len}_m(\gamma))^2 / \area(m) \geq K/400$.

Hence, we have that either ${\rm VEL}(\Gamma_{\rho,v})\ge K/400$ or ${\rm VEL}(\Gamma_{v}^\partial) \ge K/400$. \qed

\subsection{Proof of Theorem \ref{thm:nomacroscopic}} \label{sec:matingproof}
We will adapt the argument used to prove Theorem \ref{thm:converse}. In that argument, the smallness of all the circles led to a lower bound on vertex extremal length. In the present setting, results from \cite{GMS17b} show that with high probability in the a priori embedding of the mating-of-trees map, all of the cells are small and the sum of the squared diameters of the cells in a sufficiently nice region is proportional to the area of the region. We use this information to repeat the argument from Section \ref{sec:converse-proof} and get a very similar VEL lower bound, which plugs into Theorem \ref{thm:criterion} to complete the proof.

From the definition of the a priori embedding, it is immediate that given any path $(v_0,v_1,\ldots,v_k)$ in $\mcl G^\e$ there is a continuous curve in the plane that passes in order through the cells $H_{v_j}^\e$. The following lemma shows that one can also translate line segments in the plane into paths in $\mcl G^\e$.

\begin{lemma} \label{lem:pathtranslate}
The a priori embedding $\{H_v^\e\}_{v \in \mcl V \mcl G^\e}$ almost surely has the following property. Given any line segment $z_1 z_2$ in the plane, let $v_1, v_2 \in \mcl V \mcl G^\e$ satisfy $z_1 \in H_{v_1}^\e$ and $z_2 \in H_{v_2}^\e$. Then there is a path in $\mcl G^\e$ from $v_1$ to $v_2$ whose vertices are all in the set $\mcl V \mcl G^\e(z_1 z_2)$.
\end{lemma}

\begin{proof}
This statement is proved in \cite{GMS17} (see the section ``Connectivity along lines'' in the proof of Proposition 3.1) for horizontal and vertical lines. The proof extends without change to the case of lines at any angle.
\end{proof}

We now quote two results from \cite{GMS17b}. For each $v\in \mcl V \mcl G^\e$ we write $\diam(H_v^\e)$ for the diameter of the cell $H_v^\e$ associated with $v$, as described in Section \ref{sec:intromating}.

\begin{thm}\label{thm:GMSmaxdiam} There exist constants $c,q>0$ (depending only on the parameter $\gamma\in(0,2)$ discussed in Section \ref{sec:intromating}) such that for $\e$ sufficiently small,
$$ \Prob \Big ( \diam(H_v^\e) \leq \e^q \quad \forall \,\, v \in \mcl V \mcl G_1^\e \Big ) \geq 1-\e^c \, .$$
\end{thm}
\begin{proof}
Lemma 2.7 of \cite{GMS17b} gives the same estimate for all $v \in \mcl V \mcl G^\e(B(0,r))$ for some fixed $r > 1$ (recalling the renormalization by a factor of $2$ in Section \ref{sec:intromating}). On that event, every $v \in \mcl V \mcl G_0^\e$ has $\diam(H_v^\e) \leq \e^q$ and $H_v^\e \subset B(0, 1+\e^q)$. If $w \in \mcl V \mcl G_1^\e \setminus \mcl V \mcl G_0^\e$, then $w$ is inside a cycle in $\mcl G^\e$ of vertices in $\mcl V \mcl G_0^\e$, thus the connected component of $w$ in the subgraph of $\mcl G^\e$ induced by $\mcl V \mcl G^\e \setminus \mcl V \mcl G_0^\e$ is finite. On the other hand, by Lemma \ref{lem:pathtranslate}, a single connected component of this subgraph contains all the vertices $u \in \mcl V \mcl G^\e$ such that $H_u^\e$ intersects the complement of $B(0, 1+\e^q)$. Indeed, any two such cells $H_u^\e, H_{u'}^\e$ can be joined by a piecewise linear curve that avoids the disk $B(0, 1+\e^q)$. It follows that $H_w^\e \subset B(0, 1+\e^q) \subset B(0,r)$ and so the diameter estimate also holds for $H_w^\e$.
\end{proof}

To state the second result from \cite{GMS17b}, we recall the notation $A_z(r_1,r_2)$ from Section \ref{sec:converse-proof}.
For any open set $D$ we write $\area(D)$ for its Lebesgue measure. Let $q>0$ be the constant from Theorem \ref{thm:GMSmaxdiam} and set $\DD' = \{z : |z|\leq 1+ \e^q\}$. Note that on the event of Theorem \ref{thm:GMSmaxdiam}, every $v \in \mcl V \mcl G_1 ^\e$ has $H_v^\e \subset \DD'$.

\begin{thm}\label{thm:GMS} There exist constants $A,c>0$ (depending only on $\gamma$) such that when $\e$ is sufficiently small, for any $|z|\leq 1-\e^c$ and $r_1,r_2$ satisfying $1\geq r_2 \geq r_1 + \e^c$ we have that
$$ \Prob \Big ( \sum_{v \in \mcl V \mcl G ^\e \,:\, H_v^\e \subset A_z(r_1,r_2) \cap \DD'} \diam^2(H_v^\e) \leq A \cdot \area\big ( A_z(r_1,r_2) \cap \DD'\big) \Big ) \geq 1-\e^c \, .$$
\end{thm}
\begin{proof}
Follows directly from Proposition 2.9 of \cite{GMS17b} by taking $D=A_z(r_1,r_2) \cap \DD'$, noting that $\area(D)$ is at least of order $\e^c$ under our assumptions on $z,r_1,r_2$.
\end{proof}

We now begin the proof of Theorem \ref{thm:nomacroscopic}. We will show that when $\e$ is sufficiently small, with probability at least $1-\e^c$ the map $\mcl G_1^\e$ satisfies
\begin{align}\vel(\Gamma_{\rho,v}) &\geq c \log(1/\e) \quad {\rm for\ all} \quad v \in \mcl V \mcl G_1^\e \cap \mcl V \mcl G^\e ( \{ |z| \geq 1/2 \} ) \label{nomacro.goal} \\
\vel(\Gamma_{v}^\partial) &\geq c \log(1/\e) \quad {\rm for\ all} \quad v \in \mcl V \mcl G^\e ( \{ |z| \leq 1/2\} )\, , \label{nomacro.goal2} \end{align}
for some fixed constant $c>0$. By our construction in Section \ref{sec:intromating} we have that $\mcl G_2^\e$ is a submap of $\mcl G_1^\e$, hence \eqref{nomacro.goal} for $\mcl G_1^\e$ implies the same bound for $\mcl G_2^\e$. Furthermore, since both maps have the same outer face, \eqref{nomacro.goal2} for $\mcl G_1^\e$ again implies the same estimate for $\mcl G_2^\e$. Thus, by Theorem \ref{thm:criterion}, these two assertions conclude the proof and it suffices to prove \eqref{nomacro.goal} and \eqref{nomacro.goal2} for $\mcl G_1^\e$.

In fact, we prove only \eqref{nomacro.goal}; the proof of \eqref{nomacro.goal2} is very similar. Let $A,c,q>0$ be the constants from Theorems \ref{thm:GMSmaxdiam} and \ref{thm:GMS}. By Theorem \ref{thm:GMSmaxdiam}, with probability at least $1-\e^c$ we have that $\diam(H_v^\e) \leq \e^q$ for all $v \in \mcl V \mcl G_1^\e$. Let us assume this event holds.

Set $c'=\min(c,q)/4$ and $K = \lfloor c' \log_2 (1/\e) \rfloor - 3$. For $0 \leq k \leq K$, set $r_k = 2^{k+1} \e^{c'}$. Note that $r_0 = 2\e^{c'}$ and that $r_K \in [1/8,1/4]$. We would like to draw concentric circles of radius $r_0,\ldots,r_K$ around each cell $H_v^\e$, as in the proof of Theorem \ref{thm:converse}, but since we will need to use Theorem \ref{thm:GMS} in each annulus we must take care that the total number of annuli is less than $\e^{-c}$. For this reason, we specify the center points of the circles using the following procedure.

Denote $\DD_\e = \{ z : |z|\leq 1-\e^c\}$. We find points $z_1, \ldots, z_L$ in $\DD_\e$ such that
\begin{enumerate}
  \item Every cell corresponding to a vertex $v\in \mcl V \mcl G_1^\e$ is contained in $B(z_i,r_0)$ for some $1 \leq i \leq L$, and
  \item $L \leq 16\e ^{-c/2}$.
\end{enumerate}
Indeed, we take $\{z_1, \ldots, z_{L}\}$ to be an $\e^{c'}/2$-net in $\DD_\e$, that is, for any $z\in \DD_\e$ there is some $1 \leq i\leq L$ such that $|z - z_i| \leq \e^{c'}/2$ and $B(z_i,\e^{c'}/4) \cap B(z_j,\e^{c'}/4) = \varnothing$ for any $i\neq j$. Such a set of points can be easily obtained greedily: at each stage, if there is a $z\in \DD_\e$ such that $B(z,\e^{c'}/4) \cap B(z_i,\e^{c'}/4) = \varnothing$ for all $i$, then we add $z$ to the current set of $z_i$'s. Since the added disks are disjoint, this process must end after at most $16 \e^{-2c'}$ steps, so by our choice of $c'$ we have that $L \leq 16\e^{-c/2}$. Due to the maximality of this set of points, we obtain that $\DD_\e \subset \cup_{1 \leq i \leq L} B(z_i,\e^{c'}/2)$ and since $c' < \min(c,q)$ we have that $\DD' \subset \cup_{1 \leq i \leq L} B(z_i,\e^{c'})$. Lastly, for all $v\in \mcl V \mcl G_1^\e$, since $\diam(H_v^\e) \leq \e^q$ we can find $i$ such that $H_v^\e \subset B(z_i,2\e^{c'})$.  We conclude that both (1) and (2) hold for the point set $\{z_1, \ldots, z_L\}$.

We now apply Theorem \ref{thm:GMS} $LK$ times and obtain that with probability at least $1 - LK \e^c$, 
\begin{equation}\label{eq:forvelbound} \sum_{v \in \mcl V \mcl G ^\e \,:\, H_v^\e \subset A_{z_i}(r_{k-1},r_k) \cap \DD'} \diam^2(H_v^\e) \leq A \cdot \pi r_k^2  \qquad 1 \leq i \leq L,\ 1 \leq k \leq K \, .\end{equation}

Assume that \eqref{eq:forvelbound} holds and let $v \in \mcl V \mcl G_1^\e \cap \mcl V \mcl G^\e ( \{ |z| \geq 1/2 \} )$. Then $H_v^\e \subset B(z_i,r_0)$ for some $1 \leq i \leq L$ by the construction above. Let $B_k$ be the set of vertices $u \in \mcl V \mcl G_1^\e$ for which $H_u^\e \subset A_{z_i}(r_{k-1},r_k)$. We define $m: \mcl V \mcl G_1^\e \to [0,\infty)$ by
\[
m(u) = \begin{cases} \diam(H_u^\e)/r_k & \text{if } u \in B_k \text{ for some } 1 \leq k \leq K \\
0 & \text{otherwise}. \end{cases}
\]
Due to \eqref{eq:forvelbound}, the contribution to $\area(m)$ from each $B_k$ is at most $\pi A$, hence $\area(m) \leq \pi AK = O(\log(1/\e))$.

Let $\gamma$ be any path in $\mcl G_1^\e$ from $v$ to $v$ that has winding number $1$ around $\rho$. If we write $\gamma = (v_0,v_1,\ldots,v_m)$ with $v_0 = v_m = v$, then we may draw a closed curve in the plane that starts and ends at a point $w \in H_v^\e$ and passes in order through the cells $H_{v_j}^\e$. We now argue that this curve must cross each of the annuli $A_{z_i}(r_{k-1},r_k)$ for $1\leq k \leq K$. The cell $H_\rho^\e$ either contains the origin or is within distance $2\e^q$ of the origin in case the cell containing the origin was collapsed during the transition from $\mcl G_1^\e$ to $\mcl G_2^\e$. Thus we may choose $z \in H_\rho^\e$ with $|z| \leq 2\e^q$. Let $w' = -2w/|w|$ be on the opposite side of the origin from $w$ with $|w'| = 2$. We use Lemma \ref{lem:pathtranslate} to convert the straight line $zw'$ into a path in $\mcl G^\e$ from $\rho$ to $\mcl V \mcl G^\e \setminus \mcl V \mcl G_1^\e$. Since $\gamma$ has winding number $1$ around $\rho$, it separates $\rho$ from $\mcl V \mcl G^\e \setminus \mcl V \mcl G_1^\e$. (Here we used the property that every vertex in $\mcl V \mcl G^\e$ which is enclosed by a cycle in $\mcl G_1^\e$ must be an element of $\mcl V \mcl G_1^\e$.) Therefore, this new path intersects $\gamma$ at some $v_j$ and it follows that the line $zw'$ passes through $H_{v_j}^\e$. Since $|w| \geq 1/2 - \e^q$ and $r_K \leq 1/4$, in order to visit the cell $H_{v_j}^\e$ the curve induced by $\gamma$ must cross all of the annuli.

By our choice of $m$, for each $k$ the vertices $u \in B_k$ contribute at least $1/2 - O(\e^{q-c'}) \geq 1/4$ to len$_m(\gamma)$. Hence len$_m(\gamma) = \Omega(\log(1/\e))$ and we conclude that $\vel(\Gamma_{\rho,v}) = \Omega(\log(1/\e))$. This confirms \eqref{nomacro.goal}. The proof of \eqref{nomacro.goal2} is very similar; we omit the details, mentioning only that if $u \in \mcl V \mcl G_1^\e$ is a vertex of the outer face then $H_u^\e$ must intersect $\partial \DD$. \qed

\section*{Acknowledgements}
We thank Ewain Gwynne, Jason Miller and Scott Sheffield for useful discussions and their permission to use Figure \ref{fig:mated-crt-map}. This research is supported by ISF grants 1207/15 and 1707/16 as well as ERC starting grant 676970 RANDGEOM. The second author is supported by a Zuckerman Postdoctoral Fellowship.

\footnotesize{
\bibliographystyle{abbrv}
\bibliography{VEL-macroscopic}
 }

\end{document}